\documentclass{article}

\usepackage[style=trad-plain,backref=true,url=false,isbn=false,alldates=year,sortcites=true,maxnames=9]{biblatex}
\renewbibmacro{in:}{%
  \ifentrytype{article}{}{\printtext{\bibstring{in}\intitlepunct}}}

\usepackage{mathptmx}
\usepackage{amssymb,amsmath,amsthm,amsfonts}
\usepackage{graphicx}
\usepackage{bm}
\usepackage{xcolor}
\usepackage{mathrsfs}
\usepackage[shortlabels]{enumitem}
\usepackage{caption}

\usepackage{MnSymbol}
\definecolor{linkblue}{HTML}{003d73}
\definecolor{linkgreen}{HTML}{006161}
\definecolor{linkred}{HTML}{a11950}
\usepackage{hyperref}
\hypersetup{
	pdftitle={Optimization and the Topology of Spaces of Parseval Frames},
	pdfauthor={Anthony Caine, Tom Needham, and Clayton Shonkwiler},
	pdfsubject={frame theory},
	pdfkeywords={Parseval frames, full spark, optimization, frame homotopy},
	colorlinks=true,
	linkcolor=linkblue,
	citecolor=linkgreen,
	urlcolor=linkred
}
\usepackage{pifont}
\usepackage[margin=1in]{geometry}
\usepackage{mathtools}
\usepackage{booktabs}
\usepackage{authblk}
\usepackage{mathdots}

\usepackage{enumitem}

\usepackage[nameinlink,capitalize]{cleveref}

\newtheorem{thm}{Theorem}[section]
\newtheorem{lem}[thm]{Lemma}
\newtheorem{prop}[thm]{Proposition}
\newtheorem{cor}[thm]{Corollary}
\theoremstyle{definition}
\newtheorem{definition}[thm]{Definition}

\newtheorem{remark}[thm]{Remark}
\newtheorem{terminology}[thm]{Terminology}
\newtheorem{conditions}[thm]{Conditions}

\AddToHook{env/lem/begin}{\crefalias{thm}{lem}}
\AddToHook{env/prop/begin}{\crefalias{thm}{prop}}
\AddToHook{env/cor/begin}{\crefalias{thm}{cor}}
\AddToHook{env/definition/begin}{\crefalias{thm}{definition}}
\AddToHook{env/conj/begin}{\crefalias{thm}{conj}}
\AddToHook{env/remark/begin}{\crefalias{thm}{remark}}
\AddToHook{env/terminology/begin}{\crefalias{thm}{terminology}}
\AddToHook{env/conditions/begin}{\crefalias{thm}{conditions}}

\crefname{thm}{theorem}{theorems}
\crefname{lem}{lemma}{lemmas}
\crefname{prop}{proposition}{propositions}
\crefname{cor}{corollary}{corollaries}
\crefname{definition}{definition}{definitions}
\crefname{conj}{conjecture}{conjectures}
\crefname{remark}{remark}{remarks}
\crefname{terminology}{terminology}{terminologies}
\crefname{conditions}{conditions}{conditions}

\makeatletter
\renewcommand*\env@matrix[1][*\c@MaxMatrixCols c]{%
  \hskip -\arraycolsep
  \let\@ifnextchar\new@ifnextchar
  \array{#1}}
\makeatother

\newcommand{\R}{\mathbb{R}}
\newcommand{\Z}{\mathbb{Z}}
\newcommand{\N}{\mathbb{N}}
\newcommand{\Q}{\mathbb{Q}}

\newcommand{\C}{\mathbb{C}}

\newcommand{\K}{\mathbb{K}}

\newcommand{\Id}{\mathbb{I}}

\newcommand{\tr}{\operatorname{tr}}
\newcommand{\diag}{\operatorname{diag}}
\newcommand{\PF}{\operatorname{PF}}
\newcommand{\Sr}{\mathscr{S}_{\vec{r}}}
\newcommand{\lcm}{\operatorname{lcm}}

\newcommand{\potential}[1]{%
    {\def\tmp{#1}
    \ifx\tmp\empty
        {E_{\vec{r}}}
    \else
        {E_{#1}}
    \fi}}
	
\newcommand{\unitary}{\operatorname{U}}
\newcommand{\GL}{\operatorname{GL}}
\newcommand{\SL}{\operatorname{SL}}

\newcommand{\SO}{\operatorname{SO}}
\newcommand{\orthog}{\operatorname{O}}

\makeatletter
\newcommand{\subalign}[1]{%
  \vcenter{%
    \Let@ \restore@math@cr \default@tag
    \baselineskip\fontdimen10 \scriptfont\tw@
    \advance\baselineskip\fontdimen12 \scriptfont\tw@
    \lineskip\thr@@\fontdimen8 \scriptfont\thr@@
    \lineskiplimit\lineskip
    \ialign{\hfil$\m@th\scriptstyle##$&$\m@th\scriptstyle{}##$\hfil\Krcr
      #1\Krcr
    }%
  }%
}
\makeatother

\DefineBibliographyStrings{english}{%
  backrefpage = {$\uparrow$},
  backrefpages = {$\uparrow$},
  page = {p\adddot},
  pages = {pp\adddot},
}

\DeclareSourcemap{
  \maps{
    \map{
      \step[fieldset=pagetotal, null]
	  \step[fieldset=pubstate, null]
    }
  }
}

\renewbibmacro*{doi+eprint+url}{%
 \iftoggle{bbx:doi}
   {\printfield{doi}}
   {}%
 \newunit\newblock
 \iftoggle{bbx:eprint}
   {%
     \iffieldundef{doi}%
       {\usebibmacro{eprint}}%
       {}}
   {}}

\DeclareFieldFormat{eprint:urn}{%
  \mkbibacro{URN}\addcolon\space
  \ifhyperref
    {\href{https://nbn-resolving.org/urn:#1}{\nolinkurl{#1}}}
    {\nolinkurl{#1}}}
\DeclareFieldAlias{eprint:URN}{eprint:urn}
\DeclareFieldFormat{eprint:hal}{%
  \mkbibacro{HAL}\addcolon\space
  \ifhyperref
    {\href{https://hal.science/#1}{\nolinkurl{#1}}}
    {\nolinkurl{#1}}}
\DeclareFieldAlias{eprint:HAL}{eprint:hal}
\DeclareFieldFormat{eprint:numdam}{%
  Numdam\addcolon\space
  \ifhyperref
    {\href{http://www.numdam.org/item/#1}{\nolinkurl{#1}}}
    {\nolinkurl{#1}}}
\DeclareFieldAlias{eprint:Numdam}{eprint:numdam}
\DeclareFieldFormat{eprint:ark}{%
  \mkbibacro{ARK}\addcolon\space
  \ifhyperref
    {\href{https://n2t.net/ark:#1}{\nolinkurl{#1}}}
    {\nolinkurl{#1}}}
\DeclareFieldAlias{eprint:ARK}{eprint:ark}
\DeclareFieldFormat{eprint:zbl}{%
  Zbl\addcolon\space
  \ifhyperref
    {\href{https://zbmath.org/#1}{\nolinkurl{#1}}}
    {\nolinkurl{#1}}}
\DeclareFieldAlias{eprint:Zbl}{eprint:zbl}
\DeclareFieldFormat{eprint:mr}{%
  \mkbibacro{MR}\addcolon\space
  \ifhyperref
    {\href{https://mathscinet.ams.org/mathscinet-getitem?mr=#1}{\nolinkurl{#1}}}
    {\nolinkurl{#1}}}
\DeclareFieldAlias{eprint:MR}{eprint:mr}

\title{Optimization and the Topology of Spaces of Parseval Frames}

\author[$\ast$]{Anthony Caine}
\author[$\dag$]{Tom Needham}
\author[$\ddag$]{Clayton Shonkwiler}
\affil[$\ast$]{School of Applied Sciences and Arts, Arizona State University, Phoenix, AZ, USA}
\affil[$\dag$]{Department of Mathematics, Florida State University, Tallahassee, FL, USA} 
\affil[$\ddag$]{Department of Mathematics, Colorado State University, Fort Collins, CO, USA}
\date{}

\begin{document}

\maketitle

\begin{abstract}
A Parseval frame is a spanning set for a Hilbert space which satisfies the Parseval identity: a vector can be expressed as a linear combination of the frame whose coefficients are inner products with the frame vectors. There is considerable interest within the signal processing community in the structural properties of the space of finite-dimensional Parseval frames whose vectors all have the same norm, or which  satisfy more general prescribed norm constraints. In this paper, we introduce a function on the space of arbitrary spanning sets that jointly measures the failure of a  spanning set to satisfy both the Parseval identity and given norm constraints. We show that, despite its nonconvexity, this function has no spurious local minimizers, thereby extending the  Benedetto–Fickus theorem to this non-compact setting. In particular, this shows that gradient descent converges to an equal norm Parseval frame when initialized within a dense open set in the associated matrix space. We then apply this result to study the topology of frame spaces.  Using our Benedetto–Fickus-type result, we realize spaces of Parseval frames with prescribed norms as deformation retracts of simpler spaces, leading to explicit conditions which guarantee the vanishing of their homotopy groups.  These conditions yield new path-connectedness results for spaces of real Parseval frames, generalizing the Frame Homotopy Theorem, which has seen significant interest in recent years.
\end{abstract}

\section{Introduction}

A \emph{Parseval frame} for $\mathbb{K}^d$ (with $\mathbb{K} = \mathbb{R}$ or $\mathbb{C}$) is a spanning set of vectors $\{f_1,\ldots,f_n\}$ with the property that, when representing a vector $v \in \mathbb{K}^d$ as a linear combination with respect to this collection, the coefficients can be taken to be (standard) inner products:
\begin{equation}\label{eqn:parseval_property}
v = \sum_{i=1}^n \langle v, f_i \rangle f_i.
\end{equation}
This property (known as the \emph{Parseval identity}) obviously holds for orthonormal bases, so that Parseval frames generalize the notion of orthonormal bases to allow for $n > d$; indeed, a Parseval frame is just an orthonormal basis when $n=d$.

The generalization from orthonormal bases to ``overcomplete" spanning sets is motivated by the desire for robust representations of signals: considering $\mathbb{K}^d$ as a (finite-dimensional approximation of a) space of signals, one can model measurements of a signal $v \in \mathbb{K}^d$ as inner products with the frame vectors, $\langle v, f_i \rangle$. Property \eqref{eqn:parseval_property} then intuitively says that the process of reconstructing a signal from its measurements is entirely straightforward. Of course, \eqref{eqn:parseval_property} holds if the frame is an orthonormal basis (i.e., if $n=d$), but Parseval frames with $n \gg d$ offer more accurate reconstruction when measurements are contaminated with noise or when some are lost in transmission~\cite{goyalQuantizedFrameExpansions2001,casazzaEqualnormTightFrames2003,holmesOptimalFramesErasures2004}. This robustness is an important byproduct of the inherent redundancy in the measurement acquisition process. 

One frequently considers Parseval frames with the additional constraint that the frame vectors all have the same norm. Such frames are referred to as \emph{equal norm Parseval frames}. This is a natural restriction both from a practical perspective—this provides uniformity in the ``power" of the respective measurements—and from a mathematical one—for instance, the space of equal norm frames is compact. Properties of Parseval frames are a prominent topic within the larger field of \emph{frame theory}; see~\cite{waldronIntroductionFiniteTight2018} for a recent textbook on the topic.

Fixing dimension parameters $d$ and $n$, the collection of equal norm Parseval frames with this data forms a complicated algebraic variety~\cite{dykemaManifoldStructureSpaces2006,strawnFiniteFrameVarieties2010}. The geometric and topological properties of this space have been thoroughly studied~\cite{caineNormalizingParsevalFrames2024,needhamSymplecticGeometryConnectivity2021,needhamToricSymplecticGeometry2022,cahillAlgebraicGeometryFinite2013,cahillConnectivityIrreducibilityAlgebraic2017,mareConnectivityPropertiesSchur2024}, frequently with a view toward associated optimization problems. This paper extends the study of the space of equal norm Parseval frames (and related spaces) through techniques of Geometric Invariant Theory (GIT)~\cite{mumfordGeometricInvariantTheory1994,thomasNotesGITSymplectic2005}, which have been shown to be effective in this domain in the authors' previous work~\cite{mixonThreeProofsBenedetto2023,needhamFusionFrameHomotopy2023,needhamGeometricApproachesMatrix2024}. We now explain more details about our main contributions and give an outline of the paper.

\paragraph{Extension of the Benedetto–Fickus theorem (\Cref{thm:main}).}

Consider the problem of constructing an equal norm Parseval frame: this is a nontrivial task, due to the complicated structure of the space of frames, an algebraic variety of high codimension in the ambient matrix space. One natural idea would be to start with a collection of equal norm vectors (easy to generate, uniformly at random) and to then evolve them towards a Parseval frame while maintaining the equal norm condition. This could be accomplished, say, via gradient descent with respect to some loss function measuring the failure of property  \eqref{eqn:parseval_property} for a given collection. An immediate concern about the viability of this plan is that the gradient descent process could converge to a spurious local minimum which is not actually a Parseval frame. This is addressed by making a principled choice of loss function, as is shown in the celebrated \emph{Benedetto–Fickus Theorem}~\cite{benedettoFiniteNormalizedTight2003}, which serves as a main inspiration for the present paper. Roughly, ~\cite{benedettoFiniteNormalizedTight2003} introduces a certain loss function (the \emph{frame potential}) on the space of equal norm frames whose global minima are equal norm Parseval frames, and then shows that this loss function has no spurious local minima. This result has been generalized, refined, and extended in several recent articles~\cite{mixonThreeProofsBenedetto2023,
caineNormalizingParsevalFrames2024,
needhamFusionFrameHomotopy2023,
needhamGeometricApproachesMatrix2024}.

In this paper, we define a loss function on $\mathbb{K}^{d \times n}$ (a point in which is identified with a collection of $n$ vectors in $\mathbb{K}^d$) that measures, roughly, how far an arbitrary collection of vectors is from being an equal norm Parseval frame. Although this loss function is non-convex, we show in \Cref{thm:main} and \Cref{cor:no spurious local mins} that it has no spurious local minimizers. Our result, in particular, implies that one can construct an equal norm Parseval frame by starting from a random collection of vectors and gradient descending via the loss function. This extension of the Benedetto–Fickus theorem is complicated by the lack of compactness in the domain of the loss function, and requires a different proof technique, based on ideas from GIT. In fact, our result applies more generally to the construction of Parseval frames with arbitrary prescribed norms, as was previously studied in~\cite{cahillConstructingFiniteFrames2013}. We note that the construction of $d \times n$ matrices with constant spectrum and prescribed column norms---which can be achieved by gradient descent, according to our result---has important applications in wireless communication~\cite{viswanathOptimalSequencesSum1999,troppFinitestepAlgorithmsConstructing2004}.

\paragraph{Topology of frame spaces (\Cref{thm:homotopy_groups}).} We apply our Benedetto–Fickus-type theorem to describe the topology of spaces of Parseval frames with  prescribed norms; in particular, we use \Cref{thm:main} to realize frame spaces as deformation retracts of spaces whose homotopy groups are easier to understand. \Cref{thm:homotopy_groups} gives conditions (depending on coefficient field, dimension, number of frames, and choice of prescribed norms) which guarantee that the homotopy groups (of various dimensions) of these frame spaces vanish. 

Our results expand on a series of recent results in the literature stemming from Larson's \emph{Frame Homotopy Conjecture}~\cite{dykemaManifoldStructureSpaces2006}, which states that the space of equal norm Parseval frames is path-connected. This conjecture was resolved by Cahill, Mixon and Strawn in~\cite{cahillConnectivityIrreducibilityAlgebraic2017}, and there have since been several subsequent generalizations and extensions~\cite{needhamFusionFrameHomotopy2023, needhamSymplecticGeometryConnectivity2021,mareConnectivityPropertiesSchur2024,needhamAdmissibilityFrameHomotopy2022}. \Cref{thm:homotopy_groups} and its corollaries give new examples of connected frame spaces over $\R$ (\Cref{cor:real frame homotopy}) and the first (to our knowledge) general results about the higher-dimensional topology of frame spaces (\Cref{cor:vanishing_topology}). We note that our results may be of more general interest: the frame spaces we consider arise in pure mathematical contexts as moduli spaces of polygons~\cite{farberHomologyPlanarPolygon2007,farberTopologySpacesPolygons2012,kapovichModuliSpacePolygons1995} and, in the complex case, as level sets of momentum maps of Hamiltonian actions on Grassmann and flag manifolds~\cite{mareConnectivityPropertiesSchur2024,needhamSymplecticGeometryConnectivity2021}; moreover, the techniques we use to demonstrate vanishing homotopy groups may be more broadly applicable to other spaces.   

\paragraph{Outline of the paper.} \Cref{sec:potential_function} introduces a family of functions, parameterized by collections of admissible vector norms, whose minimization will be shown to produce Parseval frames with these prescribed norms. Each function in this family is referred to as a \emph{total frame energy}. We give necessary background and state our main theorem, \Cref{thm:main}. The proof of the theorem follows the general strategy introduced and exploited in our previous work~\cite{mixonThreeProofsBenedetto2023,needhamFusionFrameHomotopy2023,needhamGeometricApproachesMatrix2024,caineNormalizingParsevalFrames2024}, which is based on establishing a few claims about the structure of the total frame energy; these claims are then proved in \Cref{sec:proof_of_main_theorem}. Topological applications of \Cref{thm:main} are explored in \Cref{sec:topology}, culminating in our results on the vanishing of homotopy groups of the Parseval frame spaces, \Cref{thm:homotopy_groups} and its corollaries. The paper concludes with \Cref{sec:discussion}, a short discussion section which includes some potential future directions.

\section{The Total Frame Energy}\label{sec:potential_function}

In this section, we introduce a function on the space of arbitrary collections of vectors whose global minima are equal norm Parseval frames, and formulate the main result of the paper. We begin by fixing notation and recording some basic results of frame theory.

\subsection{Background and Notation}
\label{sec:background}

Let $\K = \R$ or $\C$ and let $d$ be a positive integer. A \emph{frame} for $\K^d$ is simply a spanning set of vectors $\{f_1,\ldots,f_n\}$, for some $n \geq d$. Let $\K^{d \times n}$ denote the space of $d \times n$ matrices. If a matrix $F = \begin{bmatrix} f_1 & f_2 & \dots & f_n\end{bmatrix} \in \K^{d \times n}$ has full rank, then the column vectors $f_1, \dots , f_n$ for $F$ define a frame for $\K^d$. Based on this correspondence, we identify the space of $n$-frames for $\mathbb{K}^d$ with the space of full-rank matrices in $\mathbb{K}^{d \times n}$. 

A matrix $F \in \mathbb{K}^{d\times n}$ is a \emph{Parseval frame} if $FF^\ast$ is the $d \times d$ identity matrix. Here, $F^\ast$ denotes the adjoint of $F$ (i.e., transpose or conjugate transpose, depending on $\mathbb{K}$). The matrix $FF^\ast$ is commonly referred to as the \emph{frame operator} of a frame $F$. We note that this definition agrees with the characterization \eqref{eqn:parseval_property} given in the introduction. 

We are interested in Parseval frames with prescribed frame vector norms, so we introduce some notation. If $\vec{r} = (r_1, \dots , r_n) \in \R_+^n$ is an $n$-tuple of positive numbers, let $\PF_d(\vec{r})$ be the collection of $d \times n$ Parseval frames $F$ with $\|f_i\|^2 = r_i$ for all $i=1, \dots , n$. In general there may not be any such frames, but we can characterize precisely when $\PF_d(\vec{r})$ is non-empty. Let $r_{(1)} \geq \dots \geq r_{(n)} > 0$ be the reverse-ordering of $(r_1, \dots, r_n)$; that is, $r_{(k)}$ is the $k$th largest of the $r_i$'s. 

\begin{definition}\label{def:admissibility}
	Suppose $\vec{r} \in \R_+^n$ and $\vec{\lambda} \in \R_+^d$. Say that $\vec{r}$ is \emph{$\vec{\lambda}$-admissible} if
	\[
		\sum_{i=1}^n r_i = \sum_{i=1}^d \lambda_i
	\]
	and
	\begin{equation}\label{eq:admissibility}
		\sum_{i=1}^k r_{(i)} \leq \sum_{i=1}^k \lambda_i \quad \text{for all $1 \leq k \leq d-1$}.
	\end{equation}
\end{definition}

In the language of~\cite{marshallInequalitiesTheoryMajorization2011}, the partial sum conditions in \eqref{eq:admissibility} say that $\vec{\lambda}$ \emph{majorizes} $\vec{r}$. As an easy consequence of the Schur–Horn theorem~\cite{hornDoublyStochasticMatrices1954,schurUberKlasseMittelbildungen1923} (see also~\cite{casazzaExistenceConstructionFinite2010,antezanaSchurHornTheorem2007,cahillConstructingFiniteFrames2013}), we have the following characterization of when Parseval frames of prescribed norms exist:

\begin{prop}\label{prop:admissibility}
	Suppose $\vec{r} \in \R_+^n$. Then $\PF_d(\vec{r})$ is non-empty precisely when $\vec{r}$ is $(\overbrace{1, \dots , 1}^d)$-admissible; that is, when
	\begin{equation}\label{eq:1-admissibility}
			\sum_{i=1}^n r_i = d \qquad \text{and} \qquad \sum_{i=1}^k r_{(i)} \leq k \quad \text{for all $1 \leq k \leq d-1$}.
	\end{equation}
\end{prop}

\begin{terminology}
    For the remainder of this paper, we will only be interested in $\vec{\lambda}$-admissibility when $\vec{\lambda} = (1, \dots , 1)$. Therefore, we will use \emph{admissible} as a synonym for \emph{$(1,\dots, 1)$-admissible}. In other words, $\vec{r}$ will be admissible precisely when $\vec{r}$ satisfies~\eqref{eq:1-admissibility}.
\end{terminology}

\subsection{Defining the Total Frame Energy}

As was described in the introduction, the main result of the paper shows that Parseval frames with presecribed vector norms can be constructed via gradient descent for a particular objective function. We now give a precise definition of this function. 

\begin{definition}\label{def:potential}
    Given a vector $\vec{r} \in \R^n_+$, we define the \emph{total frame energy} $\potential{} : \K^{d \times n} \to \R$ by
    \[
	   \potential{}(F) \coloneqq \|FF^\ast - \Id_d\|_{\text{Fr}}^2 + \frac{1}{4} \sum_{i=1}^n \left(\frac{\|f_i\|^2}{r_i} - 1\right)^2,
    \]
    where $\Id_d$ is the $d \times d$ identity matrix and $\| \cdot \|_{\text{Fr}}$ is the Frobenius norm.
\end{definition}

\begin{remark}\label{rem:squared norm}
    The definition of $\potential{}$ may seem arbitrary, but it is very natural from the perspective of symplectic geometry. Specifically, there is a Hamiltonian action of $\unitary(d) \times \unitary(1)^n$ on $\C^{d \times n}$ defined as in \eqref{eq:action} below and $\potential{}$ is the squared norm of a momentum map for this action. Our motivation for defining $\potential{}$ in this way is heavily inspired by Kirwan's pioneering work on such maps~\cite{kirwanCohomologyQuotientsSymplectic1984} (see also Lerman's survey~\cite{lermanGradientFlowNorm2005}), though our goal is to make our arguments as elementary as possible, so we mostly avoid appealing to this formalism. See our previous work~\cite{ehlerPaulsenProblemMade2018,needhamToricSymplecticGeometry2022,needhamFusionFrameHomotopy2023} for more on momentum maps arising in frame theory.
\end{remark}

\begin{remark}
    When $F$ is an equal-norm frame (i.e., there is some $r > 0$ so that $\|f_i\|^2 = r$ for all $i=1, \dots , n$), then 
    \[
        \potential{(r,\dots , r)}(F) = \|FF^\ast - \Id_d\|_{\text{Fr}}^2 = \|F^\ast F\|_{\text{Fr}}^2 - 2nr+d
    \]
    agrees, up to an additive constant, with the \textit{frame potential} $F \mapsto \|F^\ast F\|_{\text{Fr}}^2$ defined by Benedetto and Fickus~\cite{benedettoFiniteNormalizedTight2003}.

    Alternatively, when $F$ is a Parseval frame, 
    \[
        \potential{\left(\frac{d}{n}, \dots , \frac{d}{n}\right)}(F) = \frac{1}{4} \sum_{i=1}^n \left(\frac{n}{d} \|f_i\|^2 - 1\right)^2 = \frac{1}{4} \frac{n^2}{d^2} \sum_{i=1}^n\|f_i\|^4 - \frac{n}{4}
    \]
    agrees, up to additive and multiplicative constants, with Bodmann and Casazza's \emph{frame energy} 
    \[
    F \mapsto \sum_{j,k}\left(\|f_j\|^2-\|f_k\|^2\right)^2,
    \]
    which they showed~\cite[Proposition~3.12]{bodmannRoadEqualnormParseval2010} is equal to $2n \sum_{i=1}^n \|f_i\|^4 -2d^2$.

    In other words, the total frame energy $\potential{}$ contains elements of both the frame potential and the frame energy, so it is also reasonably natural from a frame theory perspective.
\end{remark}

Elements of $\PF_d(\vec{r})$ are precisely the zeros of $\potential{}$, so \Cref{prop:admissibility} implies:

\begin{cor}\label{cor:global minima}
	If $\vec{r} \in \R_+^n$ is admissible, then the elements of $\PF_d(\vec{r})$ are the global minima of $\potential{}$.
\end{cor}

In particular, $\left(\frac{d}{n}, \dots , \frac{d}{n}\right)$ is admissible, so the equal-norm Parseval frames are the global minima of $\potential{\left(\frac{d}{n}, \dots , \frac{d}{n}\right)}$.

\subsection{Optimizing the Total Frame Energy}

Since $\PF_d(\vec{r})$ consists precisely of the global minima of $\potential{}$, we might hope to be able to flow a random frame to an element of $\PF_d(\vec{r})$ by negative gradient flow of $\potential{}$. Indeed, this will turn out to be the case, at least when $\vec{r}$ is rational. The restriction to rational $\vec{r}$ is due to our method of proving the following theorem, which uses tools from algebraic geometry: we don't expect that this is an essential hypothesis.

In order to state the theorem, we recall one definition: a frame $F \in \K^{d \times n}$ is \emph{full spark} if every $d \times d$ minor of $F$ is non-singular. Full spark frames provide optimal reconstructions of the largest possible class of sparse signals~\cite{alexeevFullSparkFrames2012,donohoOptimallySparseRepresentation2003}. 

\begin{thm}\label{thm:main}
	Suppose $\vec{r} \in \Q_+^n$ is admissible, so that $\PF_d(\vec{r}) \neq \emptyset$. Consider the negative gradient flow $\Gamma_{\vec{r}}: \K^{d \times n} \times [0,\infty) \to \K^{d \times n}$ defined by the differential equation
	\[
		\Gamma_{\vec{r}}(F_0, 0) = F_0, \qquad \frac{d}{dt} \Gamma_{\vec{r}}(F_0,t) = -\nabla \potential{}(\Gamma_{\vec{r}}(F_0,t)),
	\]
	where $F_0 \in \K^{d \times n}$ is some initial matrix.
	
	If $F_0$ is full spark, then $\lim_{t \to \infty} \Gamma_{\vec{r}}(F_0,t)$ is an element of $\PF_d(\vec{r})$.
\end{thm}

The assumption that $F_0$ is full spark can be relaxed slightly: it is enough for $F_0$ to have property $\Sr$, defined below in \Cref{def:Property S} and fully characterized in \Cref{thm:general_unstable_condition}. As we'll see in \Cref{prop:full spark implies semistable}, having property $\Sr$ is a weaker condition than being full spark.

Note that $\potential{}$ is a real polynomial on the real vector space $\K^{d \times n}$, so the gradient flow from any point has a well-defined limit~\cite{lojasiewiczTrajectoiresGradientDune1984} and hence $F_\infty \coloneqq \lim_{t \to \infty} \Gamma_{\vec{r}}(F_0,t)$ exists.

It is well-known that almost every element of $\K^{d\times n}$ has full spark; see, for example,~\cite[Theorem~13]{alexeevFullSparkFrames2012} and~\cite[Theorem~2.6]{blumensathSamplingTheoremsSignals2009}. Indeed, the collection of spark-deficient matrices in $\K^{d \times n}$ is the zero locus of the product of determinants of $d \times d$ minors, so it has positive codimension in $\K^{d \times n}$, and hence is a null set with respect to any measure which is absolutely continuous with respect to Lebesgue measure. Therefore, \Cref{thm:main} tells us that almost every $d \times n$ matrix will flow to an element of $\PF_d(\vec{r})$ under the negative gradient flow of $\potential{}$.

In turn, this implies that there are points in $\K^{d \times n}$ arbitrarily close to any non-minimal critical point of $\potential{}$ which flow to a global minimum, so non-minimal critical points of $\potential{}$ cannot be basins of attraction of the gradient flow. Since $\potential{}$ is a polynomial defined on a Euclidean space, local minima must be basins of attraction~\cite[Theorem~3]{absilStableEquilibriumPoints2006}, so \Cref{thm:main} implies that $\potential{}$ has no spurious local minima:

\begin{cor}\label{cor:no spurious local mins}
	When $\vec{r} \in \Q_+^n$ is admissible, all local minima of $\potential{}$ are global minima.
\end{cor}

This is a direct analog of the classical Benedetto–Fickus theorem~\cite{benedettoFiniteNormalizedTight2003} in this setting.

As in our previous work~\cite{mixonThreeProofsBenedetto2023,needhamFusionFrameHomotopy2023,caineNormalizingParsevalFrames2024}, the strategy for proving \Cref{thm:main} is to identify a property $\Sr$ satisfying the following conditions:

\begin{conditions}\label{Sr conditions}
\begin{enumerate}
	\item Every full spark frame in $\K^{d \times n}$ satisfies $\Sr$ and
	\item gradient flow (and its limit) preserves $\Sr$, but
	\item non-minimizing critical points of $\potential{}$ do not satisfy $\Sr$.
\end{enumerate}
\end{conditions}

\subsection{Geometric Invariant Theory and Semi-Stability}

Property $\Sr$ will be semi-stability with respect to a certain algebraic group action on $\K^{d \times n}$. This idea is based on Mumford's \emph{geometric invariant theory} (GIT)~\cite{mumfordGeometricInvariantTheory1994}, so we give some general background (which is basically a gloss of Thomas's explanation~\cite{thomasNotesGITSymplectic2005}) before specifying to our problem.

Let $G$ be a reductive algebraic group that acts linearly on a finite-dimensional complex vector space $V$; for example, $\GL(V)$ or $\SL(V)$. A non-zero vector $v \in V$ is \emph{unstable} under the action of $G$ if the closure $\overline{G \cdot v}$ of the $G$-orbit of $v$ contains the origin. If $v$ is not unstable, it is called \emph{semi-stable}. From an algebraic geometry perspective, the point is that the unstable points are precisely those on which every $G$-invariant homogeneous polynomial evaluates to zero, so the semi-stable points are those which can potentially be distinguished by $G$-invariant homogeneous polynomials.

All orbits consist either entirely of semi-stable points or entirely of unstable points:

\begin{prop}[{see, e.g., \cite[Proposition 6]{mixonThreeProofsBenedetto2023}}]\label{prop:semistable orbit}
	Given $v \in V$ that is semi-stable, every point in $\overline{G \cdot v}$ is also semi-stable.
\end{prop}

Likewise, semi-stability is really a feature of the line containing $v$: $v$ is semi-stable if and only if $\lambda v$ is semi-stable for every $\lambda \in \C^\times$.

A \emph{one-parameter subgroup} of $G$ is a homomorphism of algebraic groups $\lambda : \C^\times \to G$, which induces a decomposition $V = \bigoplus_{i \in I} V_i$ and integer weights $w : I \to \Z$ so that, for every $i \in I$, $v \in V_i$, and $t \in \C^\times$,
\[
	\lambda(t) \cdot v = t^{w(i)}v.
\]

It follows immediately that a nonzero vector $v \in V$ is unstable under the action of $G$ if there exists a one-parameter subgroup $\lambda$ so that
\[
	\lim_{t \to 0} \lambda(t) \cdot v = 0.
\]
Much less obvious is that the converse holds:

\begin{thm}[{Hilbert–Mumford criterion~\cite{hilbertUeberVollenInvariantensysteme1893,mumfordGeometricInvariantTheory1994}}]\label{thm:hilbert-mumford}
	$v \in V \backslash\{0\}$ is unstable under the action of $G$ if and only if there exists a one-parameter subgroup $\lambda$ of $G$ so that 
	\[
		\lim_{t \to 0} \lambda(t) \cdot v = 0.
	\]
\end{thm}

With this setup, we now specify to our problem. The complex vector space will be $\C^{d \times n}$ and the algebraic group will be $G = \SL(d) \times S((\C^\times)^n)$, where $S((\C^\times)^n)$ is the group of diagonal $n \times n$ matrices with determinant 1. The action of $G$ on $\C^{d \times n}$ will be determined by $\vec{r}$: while the $\SL(d)$ action will always be by left multiplication, the $S((\C^\times)^n)$ action will depend on $\vec{r}$. 

To get an algebraic action, we need to convert $\vec{r} \in \Q_+^n$ into a list of positive integers, which we can do by clearing denominators. More precisely, if $r_i = \frac{a_i}{b_i}$ for all $i=1,\dots , n$, then define $s_i \coloneqq r_i\frac{\lcm\{b_1, \dots , b_n\}}{\gcd\{a_1, \dots , a_n\}}$. Then $\vec{s} = (s_1, \dots , s_n) \in \N^n$ and $\frac{s_i}{s_j} = \frac{r_i}{r_j}$ for all $i$ and $j$. Also, let $S \coloneqq \prod_{i=1}^n s_i$ be the product of the $s_i$.

Then we define the action of $G$ as follows: if $A \in \SL(d)$ and $D = \diag(t_1, \dots , t_n) \in S((\C^\times)^n)$, then define
\begin{equation}\label{eq:action}
	(A, D) \cdot F \coloneqq A F \left(\diag\left(t_1^{S/s_1}, \dots , t_n^{S/s_n}\right)\right)^{-1} = A F \diag \left(t_1^{-S/s_1}, \dots , t_n^{-S/s_n}\right).
\end{equation}

\begin{definition}\label{def:Property S}
	We say that $F \in \C^{d \times n}$ \emph{has property $\Sr$} if it is semi-stable under the $G$-action defined in \eqref{eq:action}.
\end{definition}

We will eventually build up to a full characterization of property $\Sr$ in \Cref{thm:general_unstable_condition}, but for now we just give an easy special case:

\begin{prop}\label{prop:zero implies unstable}
	If $F \in \C^{d \times n}$ contains zero vectors, then $F$ does not have property $\Sr$.
\end{prop}

\begin{proof}
	Up to re-ordering, we may as well assume the last column of $F$ is the zero vector; i.e., $F = \begin{bmatrix} f_1 & \dots & f_{n-1} & \vec{0} \end{bmatrix}$. Define the one-parameter subgroup $\lambda : \C^\times \to G$ by
	\[
		\lambda(t) \coloneqq \left( I , \diag\left(\frac{1}{t}, \dots , \frac{1}{t}, t^{n-1}\right)\right).
	\]
	Then
	\[
		\lim_{t \to 0}\lambda(t) \cdot F = \lim_{t \to 0}I\begin{bmatrix} f_1 & \dots & f_{n-1} & \vec{0} \end{bmatrix}\diag \left( t^{S/s_1}, \dots , t^{S/s_{n-1}}, t^{-(n-1)S/s_n}\right) = \lim_{t \to 0}\begin{bmatrix} t^{S/s_1}f_1 & \dots & t^{S/s_{n-1}}f_{n-1} & \vec{0} \end{bmatrix} = 0,
	\]
	so $F$ is unstable by the easy direction of \Cref{thm:hilbert-mumford}.
\end{proof}

\section{Proof of \texorpdfstring{\Cref{thm:main}}{Main Theorem}}\label{sec:proof_of_main_theorem}

In the following subsections, we verify that property $\Sr$, from \Cref{def:Property S}, satisfies each of the three conditions stated in \Cref{Sr conditions}; combining these then leads to the proof of \Cref{thm:main}.

\subsection{Full Spark Frames Have Property \texorpdfstring{$\Sr$}{Sr}}

The next result establishes the first item from \Cref{Sr conditions}. 

\begin{prop}\label{prop:full spark implies semistable}
    If $F \in \K^{d \times n}$ has full spark, then it has property $\Sr$.
\end{prop}

\begin{proof}
	Suppose $F \in \K^{d \times n}$. Whether $\K = \R$ or $\C$, we can interpret $F$ as an element of $\C^{d \times n}$, so the notion of semi-stability with respect to the $G$-action on $\C^{d \times n}$ makes sense.
	
	Our goal is to use \Cref{thm:hilbert-mumford} to show that $F$ is semi-stable, and hence has property $\Sr$. So let $\lambda : \C^\times \to G$ be a one-parameter subgroup. Then $\lambda(t) = (A(t),D(t)) \in \SL(d) \times S((\C^\times)^n)$ and, up to a change of basis, we may as well assume $A(t)$ is diagonal. There are integer weights $u(1), \dots , u(d), w(1), \dots , w(n) \in \Z$ with $u(1) + \dots + u(d) = 0$ and $w(1) + \dots + w(n) = 0$ so that 
	\[
		\lambda(t) \cdot F = (A(t), D(t)) \cdot F = \diag(t^{u(1)}, \dots , t^{u(d)}) F \diag(t^{-w(1)S/s_1}, \dots , t^{-w(n)S/s_n}) = \begin{bmatrix} t^{u(i)-w(j)S/s_j}f_{ij} \end{bmatrix}_{i,j},
	\]
	where $f_{ij}$ is the $(i,j)$ entry of $F$.
	
	Let $J \subset \{1, \dots , n\}$ with $|J|=d$ and let $F_J$ be the $d \times d$ minor of $F$ whose columns have indices in $J$, and let $(\lambda(t) \cdot F)_J$ be the corresponding minor of $\lambda(t) \cdot F$. Then
	\begin{equation}\label{eq:action minor}
		\det (\lambda(t) \cdot F)_J = \det\begin{bmatrix} t^{u(i)-w(j)S/s_j}f_{ij} \end{bmatrix}_{1 \leq i \leq d, j \in J} = t^{ \sum_{i=1}^d u(i) - \sum_{j \in J} w(j)S/s_j} \det F_J = t^{- \sum_{j \in J} w(j)S/s_j} \det F_J,
	\end{equation}
	since $\sum_{i=1}^d u(i) = 0$. As $F$ is full spark, $\det F_J  \neq 0$ for all $J$, so \eqref{eq:action minor} goes to zero as $t \to 0$ if and only if $\sum_{j \in J} w(j)\frac{S}{s_j}<0$. 
	
	Up to reindexing, we may as well assume that  $w(1) \geq \dots \geq w(n)$. Since $\sum_{j=1}^n w(j) = 0$, either we have $w(1) = \dots = w(n) = 0$, in which case we always have $\sum_{j \in J} w(j)\frac{S}{s_j} = 0$, so $\lim_{t \to 0} \det (\lambda(t) \cdot F)_J \neq 0$, or we have $w(1) > 0 > w(n)$. In this latter case, let $J = \{1, \dots , d\}$, let $\overline{s} = \min \left\{\frac{S}{s_j} : j \in J\right\}$ and consider
	\[
		\sum_{j \in J} w(j)\frac{S}{s_j} = \sum_{j=1}^d w(j) \frac{S}{s_j} \geq \sum_{j=1}^d w(j) \overline{s} = \overline{s}\sum_{j=1}^d w(j) >0
	\]
	since $\sum_{j=1}^d w(j) > \sum_{j=1}^n w(j) = 0$ and $\overline{r} > 0$. Hence, for this choice of $J$ we have
	\[
		\lim_{t \to 0} \det (\lambda(t) \cdot F)_J = \lim_{t \to 0} t^{- \sum_{j \in J} w(j)S/s_j} \det F_J \neq 0.
	\]
	Of course, this implies that $\lim_{t \to 0} \lambda(t) \cdot F \neq 0$. Since our initial choice of one-parameter subgroup $\lambda$ was arbitrary, \Cref{thm:hilbert-mumford} implies that $F$ has property $\Sr$.
\end{proof}

\subsection{Gradient Flow Preserves \texorpdfstring{$\Sr$}{Sr}}

Our next goal is to verify the second condition from \Cref{Sr conditions}. A straightforward computation gives the gradient of $\potential{}$:

\begin{prop}\label{prop:gradient}
The gradient of the total frame energy is given by
	\[
		\nabla \potential{}(F) = F \left[ 4F^\ast F+  \diag\left(\frac{\|f_1\|^2}{r_1^2}-\frac{1}{r_1} - 4, \dots , \frac{\|f_n\|^2}{r_n^2}-\frac{1}{r_n} - 4\right)\right].
	\]
\end{prop}

Define 
\[
	a_F \coloneqq \frac{\sum_{i=1}^n \left( \frac{\|f_i\|^2}{r_i^2} - \frac{1}{r_i}-4 \right)}{\sum_{i=1}^n s_i}.
\]
If we let $\Delta = -\frac{1}{S}\diag\left(s_1 \left(\frac{\|f_1\|^2}{r_1^2}-\frac{1}{r_1} - 4 -a_F \right), \dots , s_n \left(\frac{\|f_n\|^2}{r_n^2}-\frac{1}{r_n} - 4 - a_F\right)\right)$, then $\tr \Delta = 0$ and
\[
	\nabla \potential{}(F)  = \left. \frac{d}{dt} \right|_{t=0} \left(\exp\left(t\left(4FF^\ast+a_F\Id_d\right)\right),\exp(t\Delta)\right) \cdot F.
\]
Since exponentiating any square matrix gives an invertible matrix, we know that $\exp\left(t\left(4FF^\ast+a_F\Id_d\right)\right) \in \GL(d)$. Also, $\tr \Delta = 0$ implies that $\exp(t \Delta) \in S((\C^\times)^n)$, so the above equation shows that $\nabla \potential{}(F)$ is tangent to the $\left(\GL(d) \times S((\C^\times)^n)\right)$-orbit of $F$. In general, for $(A,D) \in \GL(d) \times S((\C^\times)^n)$ with $D = (t_1, \dots , t_n)$, we see that
\[
	\left((\det A)^{-1/d} A,D \right) \cdot F = (\det A)^{-1/d} A F \diag \left( t_1^{-S/s_1}, \dots , t_n^{-S/s_n}\right) = (\det A)^{-1/d} (A, D) \cdot F,
\]
so the $\left(\GL(d) \times S((\C^\times)^n)\right)$-orbit of $F$ is contained in the $\left(\C^\times \times G\right)$-orbit of $F$, and we conclude that $\nabla \potential{}(F)$ is tangent to the $\left(\C^\times \times G\right)$-orbit of $F$.

\begin{prop}\label{prop:flow preserves property}
	Suppose $F_0 \in \C^{d \times n}$ has property $\Sr$. Then so does $F_t \coloneqq \Gamma_{\vec{r}}(F_0,t)$ for all $t \in [0 , \infty)$ and so does $F_\infty \coloneqq \lim_{t \to \infty} F_t$.
\end{prop}

\begin{proof}
	We just saw that the gradient flow of $F_0$ is tangent to the $\left(\C^\times \times G\right)$-orbit of $F_0$, so we see that $F_t \in \left(\C^\times \times G\right)\cdot F_0$ for all $t \in [0,\infty)$. Since scaling by a nonzero number doesn't affect semistability, \Cref{prop:semistable orbit} implies that $F_t$ is semi-stable for all $t \in [0,\infty)$.
	
	As long as the scalar term stays away from zero, $F_\infty$ is a non-zero scalar multiple of an element of $\overline{G \cdot F_0}$, and hence \Cref{prop:semistable orbit} implies that $F_\infty$ is also semi-stable. 
	
	Therefore, we need only rule out the possibility that the scalar term goes to zero. In that case, $F_\infty$ is the zero matrix. As we show in \Cref{lem:zero is a max} below, the zero matrix is a local maximum of $\potential{}$, so nothing flows to it except the zero matrix itself. But $F_0$ has property $\Sr$, so cannot be the zero matrix by \Cref{prop:zero implies unstable}. Therefore, the scalar term cannot go to zero, and the proof is complete.
\end{proof}

\begin{lem}\label{lem:zero is a max}
	The zero matrix is a local maximum of $\potential{}$.
\end{lem}

\begin{proof}
	We prove this using the second derivative test. If we write $F = \begin{bmatrix} x_{ij} + \sqrt{-1}\, y_{ij}\end{bmatrix}_{ij}$ and we write the Hessian of $\potential{}$ as a square matrix with respect to the $x_{ij}$ and $y_{ij}$ real coordinates, then by inspection of the gradient formula from \Cref{prop:gradient} it is clear that the Hessian at $0$ of $\potential{}$ is a diagonal matrix with entries of the form $-4 - \frac{1}{r_i}$ on the diagonal. This is certainly negative definite, so $0$ is a local maximum of $\potential{}$.
\end{proof}

In fact, we can leverage this result to show that the negative gradient flow of $\potential{}$ cannot shrink any columns to the zero vector. We will prove this by showing that the negative gradient flow of $\potential{}$ causes the norm of the $i$th column to increase in a small enough neighborhood of any critical point of $\potential{}$ which has zero $i$th column.

To that end, define $g_i : \C^{d \times n} \to \R$ by $F \mapsto \|f_i\|^2$, where $f_i$ is the $i$th column of $F$. 

\begin{lem}\label{lem:cp with zero column}
	Suppose $\widehat{F} \in \C^{d \times n}$ is a critical point of $\potential{}$ whose $i$th column is the zero vector. Then there exists $\epsilon > 0$ so that, for any $F$ in an $\epsilon$-neighborhood of $\widehat{F}$,
	\[
		\langle \nabla \potential{}(F) , \nabla g_i(F) \rangle \leq 0,
	\] 
	with equality if and only if the $i$th column of $F$ is zero.
\end{lem}

\begin{proof}
	First, $\nabla g_i(F)$ is the $d \times n$ matrix with all zero columns except the $i$th, which is $2f_i$. Therefore, for any $Y \in \C^{d \times n}$, the (real) Frobenius inner product
	\begin{equation}\label{eq:dot with grad g_i}
		\langle Y, \nabla g_i(F) \rangle = \operatorname{Re}(Y^\ast \nabla g_i(F)) = 2\operatorname{Re}(y_i^\ast f_i) = 2\langle y_i, f_i\rangle.
	\end{equation}
    This implies $\langle \nabla \potential{}(F), \nabla g_i(F) \rangle = 0$ whenever the $i$th column of $F$ is zero.

	Now, we consider $F$ in a neighborhood of $\widehat{F}$ and write $F = \widehat{F} + X$ for small $X \in \C^{d \times n}$. Then, since $\widehat{F}$ is a critical point of $\potential{}$,
	\[
		\nabla \potential{}(F) = \nabla \potential{}(\widehat{F} + X) = \nabla \potential{}(X).
	\]
	Since $X$ is in a small neighborhood of the zero matrix, expanding $\nabla \potential{}$ at 0 yields
	\[
		\nabla \potential{}(X) = \nabla \potential{}(0 + X) = \nabla \potential{}(0) + H(0) X + O(\|X\|^2) = H(0) X + O(\|X\|^2),
	\]
	where $H(0)$ is the Hessian of $\potential{}$ at 0. We know from the proof of \Cref{lem:cp with zero column} that, focusing on the $i$th column, $H(0)$ simply scales the $i$th column of $X$ by $-4 - \frac{1}{r_i}$. Since $F = \widehat{F} + X$ and the $i$th column of $\widehat{F}$ is zero, the $i$th column of $X$ is the same as the $i$th column of $F$. Therefore, 
	\[
		\left(\nabla \potential{}(F)\right)_i = \left(\nabla \potential{}(X)\right)_i \approx \left(-4-\frac{1}{r_i}\right)x_i = \left(-4-\frac{1}{r_i}\right)f_i
	\]
	to first order. 
	
	Therefore, substituting $\nabla \potential{}(F)$ into \eqref{eq:dot with grad g_i}, we see that
	\[
		\langle \nabla \potential{}(F) , \nabla g_i(F) \rangle = 2 \langle \left(\nabla \potential{}(F)\right)_i, f_i \rangle \approx 2\left(-4-\frac{1}{r_i}\right)\|f_i\|^2.
	\]
	This is strictly negative whenever $f_i$ is nonzero and $F$ is close enough to $\widehat{F}$.
\end{proof}

\begin{prop}\label{prop:no new zeros}
	Suppose $F_0 \in \C^{d \times n}$ and define $F_t$ for $t \in [0,\infty]$ as in \Cref{prop:flow preserves property}. If the $i$th column of $F_t$ is zero for any $t \in [0,\infty]$, then the $i$th column of $F_0$ must also be zero.
\end{prop}

\begin{proof}
	By inspection of the formula from \Cref{prop:gradient}, observe that the $i$th column of $\nabla \potential{}(F)$ is 
	\begin{equation*}\label{eq:ith column grad}
		\left(4 FF^\ast +\left( \frac{\|f_i\|^2}{r_i^2} - \frac{1}{r_i}-4 \right)\Id_d\right) f_i.
	\end{equation*}
    Hence, the \emph{positive} gradient flow of $\potential{}$ preserves zero columns, so the \emph{negative} gradient flow cannot introduce new zero columns in finite time. This proves the result for $t \in [0,\infty)$.
	
	Now, suppose $F_\infty$ has a zero column and choose $t_0$ large enough so that $F_{t}$ satisfies the hypotheses of \Cref{lem:cp with zero column} for all $t \geq t_0$. Then \Cref{lem:cp with zero column} implies that
	\[
		\langle - \nabla \potential{}(F_t), \nabla g_i(F_t) \rangle \geq 0
	\]
	for all $t \geq t_0$, so $g_i(F_t)$ must be non-decreasing with $t$. But we know that $g_i(F_t) \geq 0$ and $0 = g_i(F_\infty) = \lim_{t \to \infty} g_i(F_t)$, so it must be the case that the $i$th column of $F_t$ is zero for all $t \geq t_0$, and hence (by the first part of the proof) that the $i$th column of $F_0$ was already zero.
\end{proof}

\subsection{Non-Minimizing Critical Points Do Not Have Property \texorpdfstring{$\Sr$}{Sr}}

Finally, we establish Condition 3 in \Cref{Sr conditions}. We begin with some necessary definitions. 

\begin{definition}\label{def:orthodecomposable}
	A frame $f_1, \dots , f_n \in \K^d$ is \emph{orthodecomposable} if there exists a nontrivial partition $P_1, \dots , P_k$ of $\{1, \dots , n\}$ and pairwise orthogonal subspaces $V_1 , \dots , V_k \subset \K^d$ with $\K^d = \bigoplus_{i=1}^k V_i$ so that, for all $j=1, \dots , k$, $\{f_i\}_{i\in P_j}$ is a frame for $V_j$.
\end{definition}

\begin{definition}\label{def:blockwiseENT}
	An orthodecomposable frame $F$ with partition $P_1, \dots , P_k$ is a \emph{blockwise (equal-norm) tight frame} if each $\{f_i\}_{i \in P_j}$ is a (equal-norm) tight frame for its span.
\end{definition}

\begin{prop}\label{prop:critical points}
	$F$ is a critical point of $\potential{}$ if and only if one of the following conditions holds:
	\begin{enumerate}
		\item \label{it:PF} $F \in \PF_d(\vec{r})$,
		\item \label{it:blockwiseTF} $F$ is a blockwise tight frame so that $\frac{1}{r_i} - \frac{\|f_i\|^2}{r_i^2}$ is constant on each block, or
		\item $F$ is the union of a frame satisfying \ref{it:PF} or \ref{it:blockwiseTF} with some zero vectors. 
	\end{enumerate}
	 In particular, the critical points of $\potential{\left(\frac{d}{n}, \dots , \frac{d}{n}\right)}$ are precisely the ENP frames, the blockwise equal-norm tight frames, and the unions of ENP frames or blockwise ENT frames with some zero vectors.
\end{prop}

\begin{proof}
From \Cref{prop:gradient}, $F$ is a critical point of $\potential{}$ if and only if 
\[
	4F F^\ast F = F\diag \left(4 + \frac{1}{r_1} - \frac{\|f_1\|^2}{r_1^2}, \dots , 4 + \frac{1}{r_n} - \frac{\|f_n\|^2}{r_n^2}\right).
\]
Considered column-wise, this says that
\[
	4FF^\ast f_i = \left( 4 + \frac{1}{r_i} - \frac{\|f_i\|^2}{r_i^2}\right) f_i,
\]
so each non-zero $f_i$ is an eigenvector of the frame operator $FF^\ast$ with eigenvalue $\frac{1}{4}\left( 4 + \frac{1}{r_i} - \frac{\|f_i\|^2}{r_i^2}\right)$. If the frame operator has distinct eigenvalues $\lambda_1 >  \dots > \lambda_k > 0$ with associated eigenspaces $E_1, \dots, E_k$, then we get a partition $P_1, \dots , P_k$ of $\{1, \dots , n\}$ with $i \in P_j$ if and only if $FF^\ast f_i = \lambda_j f_i$. 

Since eigenspaces of Hermitian matrices like $FF^\ast$ are pairwise orthogonal, we see that $F$ must be orthodecomposable with partition $P_1, \dots , P_k$. Moreover, since
\begin{equation}\label{eq:eigenvalue}
	\lambda_j = \frac{1}{4}\left( 4 + \frac{1}{r_i} - \frac{\|f_i\|^2}{r_i^2}\right)
\end{equation}
for any $i \in P_j$, it follows that $\frac{1}{r_i} - \frac{\|f_i\|^2}{r_i^2}$ is constant on $\{f_i\}_{i \in P_j}$. 
\end{proof}

We want to see that non-minimizing critical points of $\potential{}$ do not have property $\Sr$. When a critical point has some zero vectors, we can use \Cref{prop:zero implies unstable} to see that it doesn't have property $\Sr$, so we only need to worry about the case when the critical point $F$ is blockwise tight with $\frac{1}{r_i} - \frac{\|f_i\|^2}{r_i^2}$ constant on each block and $F \notin \PF_d(\vec{r})$. First, we find an alternative characterization of such frames:

\begin{lem}\label{lem:non-minimizing critical points overweight some subspace}
	Suppose $F$ is a critical point of $\potential{}$ which is blockwise tight with $\frac{1}{r_i} - \frac{\|f_i\|^2}{r_i^2}$ constant on each block and $F \notin \PF_d(\vec{r})$. Let  $P_1, \dots , P_k$ be the partition of $\{1, \dots , n\}$ and $V_1 , \dots , V_k \subset \K^d$ the pairwise orthogonal subspaces coming from the orthodecomposability of $F$. Then there exists $\ell$ so that
	\[
		\sum_{i \in P_\ell} r_i > \dim(V_\ell).
	\]
\end{lem}

\begin{proof}
	For each $j=1, \dots , k$, let $d_j = \dim(V_j)$. Up to permuting columns and changing basis, we can write $F$ in the block-diagonal form
	\[
		\begin{bmatrix} F_1 & 0 & \dots & 0 \\ 0 & F_2 & \dots & 0 \\ \vdots & \vdots & \ddots & \vdots \\ 0 & 0 & \dots & F_k \end{bmatrix},
	\]
	where $F_j$ is $d_j \times n_j$. Note that $d_1 + \dots + d_k = d$ and $n_1 + \dots + n_k = n$. 
	
	Since $\vec{r}$ satisfies~\eqref{eq:1-admissibility},
	\[
		\sum_{j=1}^k \sum_{i \in P_j} r_i = \sum_{i=1}^n r_i = d = \sum_{j=1}^k d_j,
	\]
	so either $\sum_{i \in P_j} r_i = d_j$ for all $j=1, \dots , k$, or there exists some $\ell$ so that $\sum_{i \in P_\ell} r_i > d_\ell$. The latter is our desired conclusion, so we just need to rule out the former.
	
	To that end, assume $\sum_{i \in P_j} r_i = d_j$ for all $j=1, \dots , k$. Pick an arbitrary $j$. $F$ is a blockwise tight critical point with $\frac{1}{r_i} - \frac{\|f_i\|^2}{r_i^2}$ constant on each block, so we know from~\eqref{eq:eigenvalue} that, for each $i \in P_j$,
	\begin{equation}\label{eq:ri in terms of lambda}
		\|f_i\|^2 = r_i\left(1+4r_i(1-\lambda_j)\right).
	\end{equation}
	Also, the proof of \Cref{prop:critical points} implies that
	\[
		F_jF_j^\ast = \lambda_j \Id_{d_j}.
	\]
	Hence,
	\[
		d_j \lambda_j = \tr(\lambda_j \Id_{d_j}) = \tr(F_j F_j^\ast) = \tr(F_j^\ast F_j) = \sum_{i \in P_j} \|f_i\|^2 = \sum_{i \in P_j} r_i\left(1+4r_i(1-\lambda_j)\right) = d_j + 4(1-\lambda_j) \sum_{i \in P_j}r_i^2,
	\]
	or, after simplification
	\[
		d_j(1-\lambda_j) = -4(1-\lambda_j) \sum_{i \in P_j} r_i^2.
	\]
	Since $d_j > 0$ and $\sum_{i \in P_j} r_i^2 > 0$, this is impossible unless $\lambda_j = 1$. Since our choice of $j$ was arbitrary, this is true for all $j$, meaning that $FF^\ast = \Id_d$ and, from \eqref{eq:ri in terms of lambda}, $\|f_i\|^2 = r_i$ for all $i=1,\dots , n$. In other words, $F \in \PF_d(\vec{r})$, which is precisely the condition ruled out in the hypotheses of the lemma. 
	
	So we conclude that $\sum_{i \in P_j} r_i$ cannot be equal to $d_j$ for all $j=1, \dots , k$, so there exists $\ell$ with $\sum_{i \in P_\ell} r_i > d_\ell$.
\end{proof}

With this characterization of the non-minimizing critical points of $\potential{}$, we can now show that they do not have property $\Sr$ since they satisfy the hypotheses of the following proposition.

\begin{prop}\label{prop:orthodecomposable unstable}
	Suppose $F \in \C^{d \times n}$ so that there is some subspace $V \subset \C^d$ so that $f_{i_1}, \dots , f_{i_{n_1}} \in V$ and $\sum_{j=1}^{n_1} r_{i_j} > \dim(V)$. Then $F$ does not have property $\Sr$.
\end{prop}

\begin{proof}
	Suppose $F$ is such a frame. Let $d_1 = \dim(V)$. The hypothesis 
	\[
		\sum_{j=1}^{n_1} r_{i_j} > d_1 = \frac{d_1}{d} \sum_{i=1}^n r_i
	\]
	is equivalent to
	\begin{equation}\label{eq:unstable inequality}
		\sum_{j=1}^{n_1} s_{i_j} > \frac{d_1}{d} \sum_{i=1}^n s_i.
	\end{equation}
	
	Up to permuting columns, we may as well assume that $i_j = j$ for each $j=1, \dots , n_1$, and hence, up to a change of basis, that we can represent $F$ by the block upper-triangular matrix 
	\[
		F = \begin{bmatrix} F_1 & F_2 \\ 0 & F_3 \end{bmatrix},
	\]
	where $F_1$ is $d_1 \times n_1$, $F_2$ is $d_1 \times (n-n_1)$, and $F_3$ is $(d-d_1) \times (n-n_1)$. 
	
	Note that $\|\vec{s} \|_1 = \sum_{i=1}^n s_i$ and define
	\[
		S_1 \coloneqq \sum_{i=1}^{d_1} s_i \quad \text{and} \quad \widehat{S}_1 \coloneqq \sum_{i=d_1+1}^{n} s_i = \|\vec{s}\|_1 - S_1.
	\]
	Using this notation and multiplying both sides of \eqref{eq:unstable inequality} by $d$ yields
	\begin{equation}\label{eq:unstable inequality v2}
		d S_1 > d_1 \|\vec{s}\|_1.
	\end{equation}
	
	Consider the 1-parameter subgroup $\lambda : \C^\times \to G$ given by
	\[
		\lambda(t) = \left(\begin{bmatrix} t^{(d-d_1)\|\vec{s} \|_1 S}\Id_{d_1} & 0 \\ 0 & t^{-d_1\|\vec{s} \|_1 S} \end{bmatrix}, \diag \left(t^{d\widehat{S}_1 s_1} , \dots , t^{d\widehat{S}_1 s_{d_1}}, t^{-dS_1 s_{d_1+1}}, \dots , t^{-dS_1 s_n} \right)\right).
	\]
	This really is in $G$ since 
	\[
		\det \begin{bmatrix} t^{(d-d_1)\|\vec{s} \|_1 S}\Id_{d_1} & 0 \\ 0 & t^{-d_1 \|\vec{s} \|_1 S}\Id_{d-d_1} \end{bmatrix} = t^{d_1(d-d_1)\|\vec{s} \|_1 S-(d-d_1)d_1\|\vec{s} \|_1 S} = t^0 = 1
	\]
	and
	\[
		\det \left(\diag \left(t^{d\widehat{S}_1 s_1} , \dots , t^{d\widehat{S}_1 s_{d_1}}, t^{-dS_1 s_{d_1+1}}, \dots , t^{-dS_1 s_n} \right) \right)= t^{d\widehat{S}_1S_1-dS_1 \widehat{S}_1} = t^0 = 1.
	\]
	Now,
	\[
		\lambda(t) \cdot F = \begin{bmatrix} t^{(d-d_1)\|\vec{s} \|_1 S}\Id_{d_1} & 0 \\ 0 & t^{-d_1\|\vec{s} \|_1  S}\Id_{d-d_1} \end{bmatrix} \begin{bmatrix} F_1 & F_2 \\ 0 & F_3 \end{bmatrix} \begin{bmatrix} t^{-dS \widehat{S}_1} \Id_{n_1} & 0 \\ 0 & t^{dSS_1} \end{bmatrix} = \begin{bmatrix} t^{S((d-d_1)\|\vec{s}\|_1 -d\widehat{S}_1 )}F_1 & t^{S(dS_1+(d-d_1)\|\vec{s}\|_1)}F_2 \\ 0 & t^{S(dS_1 - d_1\|\vec{s} \|_1 )} F_3 \end{bmatrix}.
	\]
	But now \eqref{eq:unstable inequality v2} implies that $dS_1 - d_1 \|\vec{s}\|_1 > 0$ and
	\[
		(d-d_1)\|\vec{s}\|_1 -d\widehat{S}_1 = (d-d_1)\|\vec{s}\|_1 -d(\|\vec{s}\|_1 - S_1) = dS_1 - d_1 \|\vec{s}\|_1 > 0,
	\]
	so we conclude that 
	\[
		\lim_{t \to 0} \lambda(t) \cdot F = 0,
	\]
	and hence $F$ does not have property $\Sr$ by \Cref{thm:hilbert-mumford}.
\end{proof}

\begin{cor}\label{cor:non-min cps don't have S}
	The non-minimizing critical points of $\potential{}$ do not have property $\Sr$.
\end{cor}

\begin{proof}
	If $F$ is a non-minimizing critical point of $\potential{}$, then it is not in $\PF_d(\vec{r})$, so it must either be blockwise tight or contain one or more zero vectors. If $F$ is blockwise tight, \Cref{lem:non-minimizing critical points overweight some subspace} and \Cref{prop:orthodecomposable unstable} imply that $F$ does not have property $\Sr$. On the other hand, if $F$ contains a zero vector, then \Cref{prop:zero implies unstable} implies it does not have property $\Sr$.
	
	Either way, we see that $F$ does not have property $\Sr$.
\end{proof}

\subsection{Completing the Proof}

Finally, we use the established fact that $\Sr$ satisfies \Cref{Sr conditions} to prove our main result.

\begin{proof}[Proof of \Cref{thm:main}]
	Suppose $F_0 \in \C^{d \times n}$ has full spark. Then, by \Cref{prop:full spark implies semistable}, $F_0$ has property $\Sr$. By \Cref{prop:flow preserves property}, this implies that $F_\infty \coloneqq \lim_{t \to \infty} \Gamma_{\vec{r}}(F_0,t)$ also has property $\Sr$. 
	
	Since $F_\infty$ is the limit point of the gradient flow, it must be a critical point; since it has property $\Sr$, \Cref{cor:non-min cps don't have S} implies that it must be a global minimizer, and hence $F_\infty \in \PF_d(\vec{r})$ by \Cref{cor:global minima}. 
	
	This proves \Cref{thm:main} when $\K = \C$. Since the gradient of $\potential{}$ at a real matrix is real, $\R^{d \times n} \subset \C^{d \times n}$ is invariant under the gradient flow, and the $\K = \R$ case follows.
\end{proof}

\section{Topological Applications}\label{sec:topology}

We now turn to some topological consequences of \Cref{thm:main}. The proof of that theorem showed that the semistable points in $\K^{d \times n}$ flow to the Parseval frames with prescribed frame vector norms; in other words, the gradient flow realizes a deformation retraction of the semistable points to $\PF_d(\vec{r})$ (see \Cref{lem:deformation_retract} below for a formal argument). Therefore, if we can show that the set of semistable points is highly connected, then $\PF_d(\vec{r})$ will be as well; this is the content of \Cref{thm:homotopy_groups}, which we prove as follows. The semistable points are the complement in $\K^{d \times n}$ of the variety of unstable points, so our strategy is to give a lower bound on the codimension of the unstable locus. The result then follows because the complement in a contractible set like $\K^{d \times n}$ of a subvariety of high codimension remains highly connected.

\subsection{Characterization of the Unstable Set}

We know from \Cref{prop:full spark implies semistable} that the unstable points are a subset of the matrices that do not have full spark. Unfortunately, the set of spark-deficient matrices is defined by a single equation, namely that the product of the determinants of all $d \times d$ minors is zero, so we generally expect this set to have codimension 1. 

However, there are many spark-deficient matrices which nonetheless have property $\Sr$; indeed, the locus of unstable points generally has codimension much bigger than 1. To see this, we need a more precise characterization of the unstable points. 

\begin{thm}\label{thm:general_unstable_condition}
    A matrix $F \in \C^{d \times n}$ fails to have property $\Sr$ if and only if there exists a subspace $V \subset \C^d$ and indices $i_1, \dots , i_k$ so that $f_{i_1}, \dots , f_{i_k} \in V$ and
    \[
        \sum_{j=1}^k r_{i_j} > \dim(V).
    \]
\end{thm}

Notice that this is a generalization of \Cref{prop:zero implies unstable}, which covers the case when $\dim(V) = 0$.

\begin{proof}
    We already proved the backward direction in \Cref{prop:orthodecomposable unstable}. As for the forward direction, assume $F$ is unstable. If the $i$th column of $F$ is the zero vector, then that column lives in the trivial subspace $\{0\}$ and $r_i > 0 = \dim \{0\}$, so the result follows. 
	
	Hence, we may as well assume that all columns of $F$ are nonzero. Since $F$ is unstable, \Cref{thm:hilbert-mumford} implies there exists a one-parameter subgroup $\lambda: \C^\times \to G$ so that $\lim_{t \to 0} \lambda(t) \cdot F = 0$. We may as well assume the first term in $\lambda(t)$ has been diagonalized, so that
	\[
		\lambda(t) = \left(\diag\left(t^{u_1}, \dots , t^{u_d}\right), \diag \left(t^{w_1}, \dots , t^{w_n}\right)\right),
	\]
	where $\sum_{i=1}^d u_i = 0$ and $\sum_{j=1}^n w_j = 0$. Then
	\[
		\lambda(t) \cdot F = \diag\left(t^{u_1}, \dots , t^{u_d}\right) \begin{bmatrix} f_{ij} \end{bmatrix}_{i,j} \diag \left(t^{-w_1 S/s_1}, \dots , t^{-w_n S/s_n}\right) = \begin{bmatrix}t^{u_i - w_j S/s_j} f_{ij} \end{bmatrix}_{i,j},
	\]
	so $\lim_{t \to 0} \lambda(t) \cdot F = 0$ implies that 
	\begin{equation}\label{eq:unstable condition}
		u_i - w_j S/s_j > 0 \quad \text{for all $i,j$ with $f_{ij} \neq 0$}.
	\end{equation} 
	
	We now break this into two cases:

\medskip

\noindent {\bf Case 1.} If $u_1 = \dots = u_d = 0$, then \eqref{eq:unstable condition} implies that $w_j < 0$ whenever the $j$th column of $F$ is not the zero vector. Since $\sum_{j=1}^n w_j = 0$, not all of the $w_j$ can be negative, so $F$ must have at least one zero column. We've already proved the result in this case.

\medskip

\noindent {\bf Case 2.} On the other hand, suppose not all $u_i$ are zero. We can assume $u_1 \geq \dots \geq u_d$ and $w_1 \geq \dots \geq w_n$, so we know that $u_1 > 0 > u_d$. Let $k$ be the largest index so that $u_k > 0$ and let $\ell$ be the largest index so that $w_\ell \geq 0$. Then for $i > k$ and $j \leq \ell$, we have that $u_i - w_jS/s_j \leq 0$, and hence \eqref{eq:unstable condition} implies $f_{ij} = 0$.
	
	Therefore, $f_1, \dots , f_\ell \in \C^k$, where we think of $\C^k \subset \C^d$ as the span of the first $k$ coordinate directions. This means we can write $F$ in the block upper-triangular form 
	\[
		F = \begin{bmatrix} F_1 & F_2 \\ 0 & F_3 \end{bmatrix}.
	\]
	If some columns of $F_3$ are zero, then we can rearrange and absorb them into $F_1$, so we may as well assume $F$ is in this form with each column of $F_3$ non-zero. Since all columns of $F_3$ contains at least one non-zero entry, we can transform $F$ into a block-diagonal form
	\[
		\widetilde{F} = \begin{bmatrix} F_1 & 0 \\ 0 & \widetilde{F}_3 \end{bmatrix}
	\]
	using only elementary row operations with determinant 1 (i.e., adding a scalar multiple of one row to another). The row operations can be achieved by left-multiplication by a matrix in $\SL(d)$, so $\widetilde{F}$ is in the same $G$-orbit as $F$, and hence must also be unstable by \Cref{prop:semistable orbit}. Since all columns of $F$ are nonzero, all the columns of $\widetilde{F}$ are also nonzero.
	
	Since $\widetilde{F}^\ast \widetilde{F}$ is block-diagonal, we see from \Cref{prop:gradient} that $\nabla \potential{}(F)$ is also block-diagonal, so gradient flow will preserve orthodecomposability, and hence $\widetilde{F}_\infty \coloneqq \lim_{t \to \infty} \Gamma_{\vec{r}}(\widetilde{F},t)$ must be an orthodecomposable critical point of $\potential{}$ with the same block structure as $\widetilde{F}$. Since $\widetilde{F}$ is unstable, \Cref{prop:flow preserves property} implies that $\widetilde{F}_\infty$ must be as well. Since $\widetilde{F}$ has no zero columns, \Cref{prop:no new zeros} implies that $\widetilde{F}_\infty$ can't have any zero columns, either, and hence \Cref{prop:critical points} tells us that $\widetilde{F}_\infty$ must be a blockwise tight critical point which is not in $\PF_d(\vec{r})$. 
	
	But then \Cref{lem:non-minimizing critical points overweight some subspace} implies that either $\sum_{j=1}^\ell r_j > k$ or $\sum_{j=\ell+1}^n r_j > d-k$; either way, $\widetilde{F}$ satisfies the conclusion of the theorem since its first $\ell$ columns lie in a $k$-dimensional subspace and its last $n-\ell$ columns lie in a $(d-k)$-dimensional subspace. 
	
	Finally, the first $\ell$ columns of $F$ lie in a $k$-dimensional subspace and, since the columns of $F$ are related to the columns of $\widetilde{F}$ by an invertible linear transformation, the last $n-\ell$ columns of $F$ must lie in some $(d-k)$-dimensional subspace since the last $n-\ell$ columns of $\widetilde{F}$ do. Therefore, $F$ also satisfies the conclusion of the theorem. 
\end{proof}

\subsection{Codimension Bound}

Let $\vec{r} \in \mathbb{Q}_+^n$ be admissible. We use $\mathcal{U}(\vec{r})$ to denote the \emph{set of unstable frames}, as is characterized in \Cref{thm:general_unstable_condition}; that is,
\[
\mathcal{U}(\vec{r}) \coloneqq \left\{F \in \C^{d \times n} \mid \, \exists \, i_1,\ldots,i_k \, \mbox{ and } V \subset \C^d \, \mbox{s.t.} \, f_{i_1},\ldots,f_{i_k} \in V \, \mbox{ and } \sum_{j=1}^k r_{i_j} > \mathrm{dim}(V)\right\}.
\]
Our goal now is to get a lower bound on the codimension of  $\mathcal{U}(\vec{r})$.

Given $\ell \in \{0,\ldots, d-1\}$, choose an indexing set $I = \{i_i,\ldots,i_k\} \subset \{1,\ldots,n\}$ with the property that $\sum_{i \in I} r_i > \ell$. We note that, by admissibility, there is always at least one such indexing set: namely, $I=\{1,\ldots,n\}$ satisfies $\sum_{i \in I} r_i = d > \ell$. The collection of available indexing sets with this property depends, in general, on the choice of $\vec{r}$ and $\ell$. We then define
\[
\mathcal{U}_{\ell,I}(\vec{r}) \coloneqq \left\{F \in \C^{d \times n} \mid \, \exists \, V \subset \C^d \, \mbox{s.t.} \, f_{i_1},\ldots,f_{i_k} \in V \mbox{ and } \mathrm{dim}(V) = \ell \right\},
\]
so that 
\[
\mathcal{U}(\vec{r}) = \bigcup_{\ell, I} \mathcal{U}_{\ell,I}(\vec{r}),
\]
where the union is over all $\ell \in \{0,\ldots,d-1\}$ and all corresponding indexing sets. 

Each $\mathcal{U}_{\ell,I}(\vec{r})$ is an algebraic variety, and our first goal is to bound the dimension of these varieties. In particular, $\mathcal{U}_{\ell,I}(\vec{r})$ is essentially a product of a  \emph{determinantal variety}, whose dimensions are well known~\cite{kleimanGeometryDeformationSpecial1971,hochsterCohenMacaulayRings1971}, with a vector space of rectangular matrices. We derive explicit dimension estimates for the sake of completeness, focusing on ``worst case" upper bounds, depending on the parameters $\ell$ and $I$. In the following, mentions of \emph{dimension} refer to dimension over the field $\C$.

Fix $\ell \in \{0,\ldots,d-1\}$ and $I = \{i_1,\ldots,i_k\}$ with $\sum_{i \in I} r_i > \ell$. To build a frame $F$ in $\mathcal{U}_{\ell,I}(\vec{r})$, one could follow three steps:
\begin{enumerate}
    \item choose an $\ell$-dimensional subspace $V$ of $\C^d$;
    \item choose $k$ vectors $f_{i_1},\ldots,f_{i_k}$ from $V$;
    \item fill out the rest of the frame with $n-k$ vectors $f_j \in \C^d$, $j \not \in I$.
\end{enumerate}
The space of $\ell$-dimensional subspaces is the Grassmann manifold, which has dimension $\ell(d-\ell)$, the selection of a single vector $f_{i_j}$ from a given subspace $V$ is an $\ell$-dimensional choice, and each choice of vector to fill the rest of the frame is a $d$-dimensional choice. Adding the dimensions involved in the construction yields
\begin{align}
\mathrm{dim}(\mathcal{U}_{\ell,I}(\vec{r})) &= \ell(d-\ell) + k \cdot \ell + (n-k) \cdot d \nonumber \\
&= \ell \cdot d - \ell^2 + n \cdot d - k(d-\ell) \label{eqn:dimension_estimate1}.
\end{align}

The expression \eqref{eqn:dimension_estimate1} is intended to emphasize the fact that the dimension is decreasing in $k$. The worst case for the upper bound is realized when $k$ is as small as possible; that is, when the chosen weights are as large as possible, $r_{(1)},\ldots,r_{(k)}$ (recall this notational convention, introduced just prior to \Cref{def:admissibility}), which is the case where the indexing set $I$ consists of the indices of the first $k$ largest entries in $\vec{r}$. Let 
\begin{equation}\label{eqn:k_hat_definition}
k_\ell(\vec{r}) \coloneqq \min \left\{k \mid \sum_{i=1}^k r_{(i)} > \ell\right\}.
\end{equation}
Together with \eqref{eqn:dimension_estimate1}, this gives
\begin{equation}\label{eqn:dimension_estimate11}
\mathrm{dim}(\mathcal{U}_{\ell,I}(\vec{r})) \leq \ell \cdot d - \ell^2 + n \cdot d - k_\ell(\vec{r}) \cdot (d-\ell).
\end{equation}

Next, we introduce the ansatz that there exists a universal (i.e., independent of $\ell$) constant $c > \frac{d}{n}$ satisfying 
\begin{equation}\label{eqn:k_hat}
k_\ell(\vec{r}) \geq c \cdot \frac{n\ell}{d}, \qquad \forall \; \ell \in \{1,\ldots,d-1\}.
\end{equation}
Such a $c$ must exist. Indeed, by the admissibility condition, we must have $k_\ell(\vec{r}) > \ell$, as we would otherwise obtain the contradiction
\[
\ell < \sum_{i=1}^{k_\ell(\vec{r})} r_{(i)} \leq \sum_{i=1}^\ell r_{(i)} \leq \ell.
\]
It follows that (assuming $d > 1$)
\[
k_\ell(\vec{r}) \geq \ell + 1 = \frac{d (\ell + 1)}{n\ell}\frac{n\ell}{d} \geq  \frac{d^2}{n(d-1)}\frac{n\ell}{d},
\]
so that $c = \frac{d^2}{n(d-1)} > \frac{d}{n}$ has the desired properties. Depending on the choice of $\vec{r}$, $c$ could be chosen to be much larger, so we leave it as a parameter to be manipulated later on.

Incorporating \eqref{eqn:k_hat} into the dimension estimate \eqref{eqn:dimension_estimate2}, we have the following bound whenever $\ell > 0$:
\begin{align}
\mathrm{dim}(\mathcal{U}_{\ell,I}(\vec{r})) &\leq \ell \cdot d - \ell^2 + n \cdot d - c n \ell + c \frac{n}{d} \ell^2 \nonumber \\
&= \big(c \frac{n}{d} - 1\big) \ell^2 - (cn-d)\ell + nd. \label{eqn:dimension_estimate2}
\end{align}
The coefficient of $\ell^2$ is positive, by our assumption on $c$, so that the worst case upper bound occurs at the extremal points of the domain of $\ell \in \{1, \ldots, d-1\}$. At both of these points (i.e., $\ell =1$ and $\ell = d-1$), \eqref{eqn:dimension_estimate2} simplifies to 
\[
-(cn-d)\frac{d-1}{d} + nd.
\]
On the other hand, if $\ell = 0$, then $k_\ell(\vec{r})  = 1$, so that \eqref{eqn:dimension_estimate11} immediately simplifies to 
\[
\mathrm{dim}(\mathcal{U}_{0,I}(\vec{r})) \leq nd - d.
\]

Putting these dimension estimates together, we have proved the following result. In the following, we emphasize the field of coefficients by writing $\mathcal{U}_{\ell,I}^\C(\vec{r}) = \mathcal{U}_{\ell,I}(\vec{r})$, $\mathcal{U}^\C(\vec{r}) = \mathcal{U}(\vec{r})$ and $\mathrm{codim}_\C$ for codimension over $\C$. 

\begin{lem}\label{lem:codimension_bound}
    Every stratum $\mathcal{U}^\C_{\ell,I}(\vec{r})$ of the unstable set $\mathcal{U}^\C(\vec{r})$ satisfies the codimension bound 
\[
\mathrm{codim}_\C(\mathcal{U}^\C_{\ell,I}(\vec{r})) \geq \min \left\{d,(cn-d)\frac{d-1}{d}\right\},
\]
with $c$ as in the ansatz \eqref{eqn:k_hat}.
\end{lem}

The reason for emphasizing the field $\C$ above is that we would like to extend our results to frames over $\R$. Accordingly, let
\[
\mathcal{U}^\R(\vec{r}) \coloneqq \mathcal{U}^\C(\vec{r}) \cap \R^{d \times n},
\]
where $\R^{d \times n} \hookrightarrow \C^{d \times n}$ is included in the obvious way. It is straightforward to show that 
\[
\mathcal{U}^\R(\vec{r}) = \left\{F \in \R^{d \times n} \mid \, \exists \, i_1,\ldots,i_k \, \mbox{ and } V \subset \R^d \, \mbox{s.t.} \, f_{i_1},\ldots,f_{i_k} \in V \, \mbox{ and } \sum_{j=1}^k r_{i_j} > \mathrm{dim}_\R(V)\right\}.
\]
Letting $\mathcal{U}_{\ell,I}^\R(\vec{r})$ denote the associated strata, the codimension bound derivations above go through mutatis mutandis. Combining this observation with  \Cref{lem:codimension_bound} leads to the following codimension bounds over the reals (treating $\C^{d}$ and $\C^{d \times n}$ as real vector spaces).

\begin{lem}\label{lem:real_codimension_bound}
    The strata of the unstable sets satisfy the codimension bounds
\begin{equation}\label{eqn:codimension}
\mathrm{codim}_\R(\mathcal{U}^\C_{\ell,I}(\vec{r})) \geq 2 \cdot \min \left\{d,(cn-d)\frac{d-1}{d}\right\},
\end{equation}
and 
\begin{equation}\label{eqn:codimension_real}
\mathrm{codim}_\R(\mathcal{U}^\R_{\ell,I}(\vec{r})) \geq \min \left\{d,(cn-d)\frac{d-1}{d}\right\},
\end{equation}
with $c$ as in \eqref{eqn:k_hat}.
\end{lem}

\begin{remark}\label{rmk:codimension_bounds_examples}
    Either possibility on the right hand side of \eqref{eqn:codimension_real} can realize the minimum, depending on the choice of $\vec{r}$. We now show this explicitly.

    Let $\vec{r} = (d/n,d/n,\ldots,d/n)$. Then we can take $c = 1$ in \eqref{eqn:k_hat}, and this choice is essentially optimal, so that 
    \[
    (cn-d)\frac{d-1}{d} = (n-d)\frac{d-1}{d} \geq d
    \]
    holds whenever 
    \[
    n \geq \frac{d(2d - 1)}{d-1}.
    \]
    In this case, the bound \eqref{eqn:codimension} becomes  $\mathrm{codim}_\R(\mathcal{U}^\R_{\ell,I}(\vec{r})) \geq d$. 

    On the other hand, let $r_1=\cdots = r_{d-1} = 1$, $r_d = 1-\epsilon$, $r_{d+1} = \cdots = r_n = \frac{\epsilon}{n-d}$, for some small $\epsilon$. Then $\vec{r}$ is admissible, and (assuming $d > 1$)
    \[
    c = \frac{d^2}{n(d-1)}
    \]
    is an essentially optimal choice of constant in \eqref{eqn:k_hat} (see the discussion following \eqref{eqn:k_hat}). In this case,
    \[
    (cn-d)\frac{d-1}{d} = \left(\frac{d^2}{n(d-1)}n - d \right)\frac{d-1}{d} = 1 \leq d,
    \]
    so that \eqref{eqn:codimension} becomes $\mathrm{codim}_\R(\mathcal{U}^\R_{\ell,I}(\vec{r})) \geq (cn-d)\frac{d-1}{d}$.
\end{remark}

\subsection{Vanishing Homotopy Groups for \texorpdfstring{$\PF_d(\vec{r})$}{PFd(r)}}

We're now ready to state and prove our result on the topology of the $\PF_d(\vec{r})$ spaces, which is stated in terms of its homotopy groups $\pi_q(\PF_d(\vec{r}))$. In particular, we give conditions which imply that this is the trivial group---if this is the case, we say that $\PF_d(\vec{r})$ is \emph{$q$-connected}. As special cases, $0$-connected is the same as being connected in the usual sense from point-set topology (which is the same as being path connected for algebraic varieties such as the frame spaces), and $1$-connected is typically referred to as being simply connected. 

In this subsection, we distinguish between Parseval frames over different fields by writing $\mathrm{PF}_d^\mathbb{K}(\vec{r})$ for the space of frames over the field $\mathbb{K} \in \{\mathbb{R},\mathbb{C}\}$. 

\begin{thm}\label{thm:homotopy_groups}
    Let $q$ be a non-negative integer, let $\vec{r} \in \mathbb{Q}^n_+$ be admissible, and let $c$ be a choice of constant satisfying \eqref{eqn:k_hat} for $\vec{r}$.
    \begin{enumerate}
        \item For real frames, suppose that 
    \[
    c \geq \frac{d}{n}\left(\frac{q+d+1}{d-1}\right) \quad \mbox{and} \quad d \geq q + 2.
    \]
    Then $\mathrm{PF}^\R_d(\vec{r})$ is $q$-connected.
    \item For complex frames, suppose that 
    \[
    c \geq \frac{d}{2n}\left(\frac{q+d}{d-1}\right) \quad \mbox{and} \quad d \geq \frac{q + 2}{2}.
    \]
    Then $\mathrm{PF}^\C_d(\vec{r})$ is $q$-connected.
    \end{enumerate}
\end{thm}

Let $\mathcal{S}^\mathbb{C}(\vec{r}) \subset \C^{d \times n}$ denote the \emph{semistable set}, i.e., the complement of the unstable set $\mathcal{U}^\C(\vec{r})$ that was discussed in the previous subsections, 
\[
\mathcal{S}^\mathbb{C}(\vec{r}) \coloneqq \mathbb{C}^{d \times n} \setminus \mathcal{U}^\C(\vec{r}).
\] 
Equivalently, $\mathcal{S}^\mathbb{C}(\vec{r})$ is the set of frames having property $\Sr$. Let $\mathcal{S}^\mathbb{R}(\vec{r})$ denote the subset of the semistable set consisting of real matrices, i.e.,
\[
\mathcal{S}^\mathbb{R}(\vec{r}) \coloneqq \R^{d \times n} \setminus \mathcal{U}^\R(\vec{r}).
\]

The proof of \Cref{thm:homotopy_groups} relies on the following lemma. 

\begin{lem}\label{lem:deformation_retract}
    Gradient descent of the total frame energy $\potential{}$ (i.e., the flow of the differential equation in \Cref{thm:main}) defines a strong deformation retract from $\mathcal{S}^\mathbb{K}(\vec{r})$ onto the space of Parseval frames $\mathrm{PF}_d^\mathbb{K}(\vec{r})$. In particular, $\mathcal{S}^\mathbb{K}(\vec{r})$ and $\mathrm{PF}_d^\mathbb{K}(\vec{r})$ are homotopy equivalent.
\end{lem}

\begin{proof}
    The complex case follows from the main result of \cite{lermanGradientFlowNorm2005} (an exposition of unpublished work of Duistermaat), which characterizes the gradient descent dynamics of squared norms of momentum maps on symplectic manifolds. This applies to the total frame energy, according to \Cref{rem:squared norm}. However, the result follows in our special case by an essentially elementary argument: the fact that gradient descent defines a well-defined map from $\mathcal{S}^\mathbb{C}(\vec{r}) \times [0,\infty]$ onto $\mathrm{PF}_d^\mathbb{C}(\vec{r})$, which fixes $\mathrm{PF}_d^\mathbb{C}(\vec{r})$, follows from \Cref{thm:main}, and the continuity of this map then follows from the analytical arguments in \cite{lermanGradientFlowNorm2005} (see page 124 therein), which do not require any machinery from symplectic geometry. The real case follows by considering $\R^{d \times n} \hookrightarrow \C^{d \times n}$ and observing that the gradient flow preserves this subspace.
\end{proof}

\begin{proof}[Proof of \Cref{thm:homotopy_groups}]
    Consider the real case of the theorem.  As homotopy groups are invariant under homotopy equivalence, \Cref{lem:deformation_retract} tells us that it suffices to prove that the assumptions on $c$, $d$ and $q$ imply triviality of the homotopy groups of $\mathcal{S}^\mathbb{R}(\vec{r})$. We have that $\mathcal{U}^\R(\vec{r})$ is a union of algebraic varieties with codimensions bounded by the minimum of $d$ and $(cn-d)\frac{d-1}{d}$. Our assumptions on $c$ and $d$ imply that the minimum is bounded below by $q+2$. By a general transversality argument (see \cite[Lemma 3.12]{needhamGeometricApproachesMatrix2024} for details), it follows that, for $d \geq q+2$, 
    \[
    \pi_q(\mathcal{S}^\mathbb{R}(\vec{r})) = \pi_q(\R^{d \times n}) = 0.
    \]
    The proof in the complex case is similar.
\end{proof}

\subsection{Consequences} 

While \Cref{thm:homotopy_groups} gives precise conditions for vanishing homotopy groups, the conditions themselves are not immediately easy to interpret. We now provide some direct corollaries of this result with more straightforward sufficient conditions. 

We first show that spaces of equal norm Parseval frames can be forced to have trivial topology in any dimension by taking a large number of frame vectors.

\begin{cor}\label{cor:vanishing_topology}
    Fix a field $\mathbb{K} \in \{\R,\C\}$, a positive integer $q$ and a dimension $d$ with
    \[
    d \geq \left\{\begin{array}{rl}
    q+2 & \mbox{if } \mathbb{K} = \R, \\
    \frac{q+2}{2} & \mbox{if } \mathbb{K} = \C.
    \end{array}\right.
    \]
    The space of equal norm Parseval frames with $n$ vectors in $\mathbb{K}^d$ is $q$-connected for all sufficiently large $n$. 
\end{cor}

\begin{proof}
    As was observed in \Cref{rmk:codimension_bounds_examples}, we can take the constant $c=1$ in \eqref{eqn:k_hat}, so that the sufficient conditions in \Cref{thm:homotopy_groups} hold by taking $n$ large enough.
\end{proof}

Next, we recover a known result from our previous paper~\cite{needhamSymplecticGeometryConnectivity2021}. We showed in the main theorem of \cite{needhamSymplecticGeometryConnectivity2021} that the space of $n$-frames in $\mathbb{C}^d$ with prescribed frame operator $S$ and admissible vector of squared vector norms $\vec{r}$ is path connected, thereby generalizing the Frame Homotopy Theorem of \cite{cahillConnectivityIrreducibilityAlgebraic2017}. \Cref{thm:homotopy_groups} provides an alternative proof of the special case where $S$ is the identity matrix, as we state in the following corollary.

\begin{cor}\label{cor:complex frame homotopy}
    The space $\mathrm{PF}_d^\C(\vec{r})$ is path-connected for any admissible $\vec{r} \in \Q_+^n$.
\end{cor}

\begin{proof}
    Assume $d > 1$, as the claim is otherwise trivial. Let $c = \frac{d}{2n} \cdot \frac{d}{d-1}$. For any $\ell \in \{1,\ldots,d-1\}$, we have 
    \[
    c \cdot \frac{n \ell}{d} = \frac{d}{2(d-1)} \ell \leq \ell.
    \]
    Since $\vec{r}$ is admissible, the quantity $k_\ell$ of \eqref{eqn:k_hat_definition} satisfies $k_\ell > \ell$, as was observed in the argument following \eqref{eqn:k_hat}. Putting this together with the estimate above, we have
    \[
    k_\ell > c \cdot \frac{n \ell}{d},
    \]
    so that our choice of $c$ satisfies \eqref{eqn:k_hat}. 

    We now apply \Cref{thm:homotopy_groups} in the case $q=0$: we have $d \geq 1 = \frac{q + 2}{2}$ and 
    \[
    c = \frac{d}{2n} \cdot \frac{d}{d-1} = \frac{d}{2n} \cdot \frac{q + d}{d-1},
    \]
    so that the theorem implies $\pi_0(\mathrm{PF}_d^\C(\vec{r}))$ is trivial, i.e., $\mathrm{PF}_d^\C(\vec{r})$ is connected. Since this space is locally path-connected, this implies that it is path-connected. 
\end{proof}

Finally, we focus on the real case and give several novel examples of spaces $\mathrm{PF}_d^\R(\vec{r})$ which are path-connected. A complete characterization of the vectors $\vec{r}$ such that $\mathrm{PF}_d^\R(\vec{r})$ is path-connected remains an open question, and the existing results to this point are relatively sparse: 
\begin{itemize}
    \item Cahill, Mixon, and Strawn~\cite{cahillConnectivityIrreducibilityAlgebraic2017} showed that when $n \geq d+2 \geq 4$ the space of equal-norm Parseval frames is path-connected, resolving Larson's Frame Homotopy Conjecture;
    \item Mare~\cite{mareConnectivityPropertiesSchur2024} showed that $\mathrm{PF}_d^\R(\vec{r})$ is path-connected for certain $\vec{r}$ with entries repeating in special patterns;
    \item Kapovich and Millson~\cite{kapovichModuliSpacePolygons1995} showed that the space of planar polygons with edge lengths given by $\vec{r}$, which can be interpreted as the quotient $\mathrm{PF}_2^\R(\vec{r})/(\SO(2) \times \orthog(1)^n)$,\footnote{For more on the connection between frames and polygons, see Copenhaver et al.~\cite{copenhaverDiagramVectorsTight2014} or Han et al.~\cite[Chapter~4]{hanFramesUndergraduates2007}. These constructions seem to originate with an idea of Goyal, Kovačević, and Kelner~\cite[\S 2.2.1]{goyalQuantizedFrameExpansions2001}, though there is a sense in which all tight frames can be interpreted as polygons~\cite{flaschkaBendingFlowsSums2005}.} is disconnected if and only if there exist $i,j,k$ such that $r_i + r_j > 1$ and $r_j + r_k > 1$ and $r_k + r_i > 1$ (recall that $r_1 + \dots + r_n = 2$ for admissible $\vec{r}$).\footnote{More generally, the homology groups of the planar polygon spaces are known~\cite{farberHomologyPlanarPolygon2007}.} This gives conditions ensuring that some of the spaces $\mathrm{PF}_2^\R(\vec{r})$ are disconnected, since a space is disconnected whenever any quotient is, but the reverse implication does not follow since quotients of disconnected spaces by disconnected groups like $\SO(2) \times \orthog(1)^n$ can be connected.
\end{itemize}

The next result establishes the existence of several infinite families of vectors $\vec{r}$ such that the corresponding space of frames $\mathrm{PF}_d^\R(\vec{r})$ is path-connected.  

\begin{cor}\label{cor:real frame homotopy}
    Let $d \geq 2$ and let $\vec{r} \in \Q_+^n$ be an admissible vector, such that the constant $c$ from \eqref{eqn:k_hat} satisfies 
    \begin{equation}\label{eqn:strong_c_condition}
    c \geq \frac{d}{n}\left(\frac{2d}{d-1}\right).
    \end{equation}
    Then there exists $\epsilon > 0$ such that, for any admissible vector $\vec{s} \in \Q_+^n$ with $|s_i - r_i| < \epsilon$ for all $i$, the space $\operatorname{PF}_d^\R(\vec{s})$ is path-connected. 
\end{cor}

The condition \eqref{eqn:strong_c_condition} holds for equal norm Parseval frames (i.e., where $r_i = \frac{d}{n}$, in which case we can take $c=1$) whenever $n \geq \frac{2d^2}{d-1}$. So this result, in particular, implies that there is an open (with respect to the subspace topology on $\Q_+^n \subset \R^n$) set of admissible vectors around the constant vector so that the associated frame space is path-connected. 

\begin{proof}
    Let $\vec{s} \in \Q_+^n$ be admissible. We claim that, for each $\ell \in \{1,\ldots,d-1\}$, there exists $\epsilon_\ell > 0$ such that $|s_i - r_i| < \epsilon_\ell$, for all $i$, implies that 
    \[
    k_\ell(\vec{s}) \geq k_\ell(\vec{r}) - 1,
    \]
    with $k_\ell$ defined as in \eqref{eqn:k_hat_definition}. Indeed, supposing that $|s_i - r_i| < \epsilon_\ell$ for some $\epsilon_\ell > 0$, and that $k_\ell(\vec{s}) < k_\ell(\vec{r}) - 1$, leads to 
    \[
    \ell < \sum_{i=1}^{k_\ell(\vec{r}) - 2} s_{(i)} \leq \sum_{i=1}^{k_\ell(\vec{r}) - 2} r_{(i)} + (k_\ell(\vec{r})-2)\epsilon_\ell, 
    \]
    which implies 
    \[
    \ell - (k_\ell(\vec{r})-2)\epsilon_\ell < \sum_{i=1}^{k_\ell(\vec{r}) - 2} r_{(i)}.
    \]
    If such a vector $\vec{s}$ existed for every choice of $\epsilon_\ell > 0$, then we would deduce that 
    \[
    \ell \leq \sum_{i=1}^{k_\ell(\vec{r}) - 2} r_{(i)} < \sum_{i=1}^{k_\ell(\vec{r}) - 1} r_{(i)},
    \]
    which gives a contradiction to the minimality of $k_\ell(\vec{r})$, so that our claim is established. 

    Now let $\epsilon > 0$ such that $|s_i - r_i| < \epsilon$ for all $i$ implies that $k_\ell(\vec{s}) \geq k_\ell(\vec{r}) - 1$ holds for all $\ell \in \{1,\ldots,d-1\}$ (e.g., take $\epsilon$ to be the minimum of the $\epsilon_\ell$'s whose existence was demonstrated above). Setting 
    \[
    c' = c - \frac{d}{n}, 
    \]
    we have, for any $\ell \in \{1,\ldots,d-1\}$,
    \[
    c' \cdot \frac{n\ell}{d} = \left(c - \frac{d}{n}\right) \cdot \frac{n\ell}{d} \leq  \left(c - \frac{d}{n\ell}\right) \cdot \frac{n\ell}{d} = c \cdot \frac{n\ell}{d} - 1 \leq k_\ell(\vec{r}) - 1 \leq k_\ell(\vec{s}).
    \]
    Moreover, we have 
    \[
    c' = c - \frac{d}{n} \geq \frac{d}{n} \left(\frac{2d}{d-1}\right) - \frac{d}{n} = \frac{d}{n}\left(\frac{d+1}{d-1}\right).
    \]
    Therefore, $c'$ satisfies the ansatz \eqref{eqn:k_hat} for the vector $\vec{s}$, as well as the bound in Theorem \ref{thm:homotopy_groups} which is sufficient for $\mathrm{PF}_d^\R(\vec{s})$ to be 0-connected, hence $\mathrm{PF}_d^\R(\vec{s})$ is path-connected. 
\end{proof}

\section{Discussion}\label{sec:discussion}

An obvious avenue for improving \Cref{thm:main} would be to show that the conclusion holds for all admissible $\vec{r} \in \R_+^n$, not just the rational ones. This will require a completely different method of proof, though, since the $G$-action we've defined here is not algebraic when $\vec{r}$ is not rational. That said, it seems likely that \Cref{cor:no spurious local mins} (our version of the Benedetto–Fickus theorem) can be proved for irrational $\vec{r}$ using approaches similar to those employed in Sections~2 and~3 of~\cite{mixonThreeProofsBenedetto2023}.

Another way of generalizing \Cref{thm:main} would be to characterize the limiting behavior of the negative gradient flow of $\potential{}$ when $\vec{r}$ is \emph{not} admissible. The dependence on $\vec{r}$ seems to be non-trivial, and we have not been able to formulate a coherent conjecture. 

In our definition of $\potential{}$, we could have replaced $\Id_d$ in the term $\|FF^\ast - \Id_d\|_{\text{Fr}}^2$ by any desired positive-definite matrix $S$. Under appropriate admissibility assumptions, does the negative gradient flow of such a modified energy function limit to frames $F$ with frame operator $F F^\ast = S$ and squared frame vector norms $\|f_i\|^2 = r_i$ for $i=1, \dots , n$? The answer seems to be yes in small-scale experiments, but there are some technical challenges in extending our proof to this case.

As with the original Frame Homotopy Conjecture, \Cref{cor:complex frame homotopy,cor:real frame homotopy} are potentially useful in that they imply that there is a way of continuously interpolating between any two Parseval frames with the same frame vector norms. With an eye to \Cref{thm:homotopy_groups} and \Cref{cor:vanishing_topology}, are there any practical consequences of $q$-connectedness of $\PF_d^\K(\vec{r})$ when $q > 0$? For example, simple connectedness implies that if we have two different interpolations between Parseval frames, we can continuously interpolate between these interpolations. Also, simple connectedness implies that all closed 1-forms are exact, so line integrals involving closed 1-forms are automatically path-independent.

\subsection*{Acknowledgments} 

This paper had its origins in conversations that took place at the Oberwolfach Mini-Workshop on Algebraic, Geometric, and Combinatorial Methods in Frame Theory in October, 2018, so we would like to thank the organizers, the Mathematisches Forschungsinstitut Oberwolfach, and especially Martin Ehler, Milena Hering, and Chris Manon, with whom we first outlined some of these ideas in~\cite{ehlerPaulsenProblemMade2018}. 

This work was supported by the National Science Foundation (DMS–2107808, Tom Needham; DMS–2107700, Clayton Shonkwiler).

\printbibliography

\end{document}